\documentclass[preprint]{elsarticle}

\usepackage{amssymb}
\usepackage{amsmath}
\usepackage{amsthm}
\usepackage{enumerate}
\usepackage{graphicx}
\usepackage{amsfonts}

\journal{...}

\newtheorem{theorem}{Theorem}[section]

\newtheorem{prop}{Proposition}[section]
\newtheorem{remark}{Remark}[section]
\newtheorem{comment}{Comment}

\newtheorem{exam}{Example}[section]
\newtheorem{problem}{Problem}[section]

\numberwithin{equation}{section}

\theoremstyle{definition}

\begin{document}

	\begin{frontmatter}

		\title{Poisson approximation to the binomial distribution: extensions to the convergence  of positive  operators }
			

\author[1]{Ana-Maria Acu${}^*$}
\author[2]{Margareta Heilmann}
\author[3]{Ioan Rasa}
\author[4]{Andra Seserman}
\address[1]{Lucian Blaga University of Sibiu, Department of Mathematics and Informatics, Str. Dr. I. Ratiu, No.5-7, RO-550012  Sibiu, Romania, e-mail: anamaria.acu@ulbsibiu.ro, telephone number: +040766258465}
\address[2]{School of Mathematics and Natural Sciences,
	University of Wuppertal,
	Gau{\ss}stra{\ss}e 20,
	D-42119 Wuppertal, Germany, 	e-mail: heilmann@math.uni-wuppertal.de}
\address[3]{Technical University of Cluj-Napoca, Faculty of Automation and Computer Science, Department of Mathematics, Str. Memorandumului nr. 28, 400114 Cluj-Napoca, Romania
	e-mail:   ioan.rasa@math.utcluj.ro }
\address[4]{Technical University of Cluj-Napoca, Faculty of Automation and Computer Science, Department of Mathematics, Str. Memorandumului nr. 28, 400114 Cluj-Napoca, Romania
	e-mail: campan$\underline{\,\,\,}$aandra@yahoo.com}

		\begin{abstract} The idea behind Poisson approximation to the binomial distribution was used in [J. de la Cal, F. Luquin, J.  Approx. Theory,
			68(3), 1992, 322-329] and subsequent papers in order to establish the convergence of suitable sequences of positive linear operators. The proofs in these papers are given using probabilistic methods. We use similar methods, but in analytic terms. In this way we recover some known results and establish several new ones. In particular, we enlarge the list of the limit operators and give  characterizations of them.
		
		\end{abstract}

		
	\end{frontmatter}
\section{Introduction} 

Let $X_n(p_n)$ be a binomial random variable with parameters $n\in {\mathbb N}$ and $p_n\in[0,1]$. Its characteristic function is 
\begin{align*} g_n(s)&=E\left(\exp (is X_n)\right)=\displaystyle\sum_{k=0}^n{n\choose k}p_n^k (1-p_n)^{n-k}\exp(isk)\\
	&=(1-p_n+p_ne^{is})^n, s\in{\mathbb R},  
	\end{align*}
where $E$ stands for mathematical expectation and  $i$ is the imaginary unit.

Let $X(\lambda)$ be a Poisson variable with parameter $\lambda\geq 0$, having the characteristic function 
$$ h(s)=\displaystyle\sum_{k=0}^{\infty} e^{-\lambda}\dfrac{\lambda^k}{k!}\exp(isk)=e^{\lambda(e^{is}-1)},\, s\in {\mathbb R}.  $$
Suppose that $\displaystyle\lim_{n\to\infty}np_n=\lambda$. Then $\displaystyle\lim_{n\to\infty} g_n(s)=h(s)$, $s\in {\mathbb R}$, and hence $X_n(p_n)$ converges in law to $X(\lambda)$. This is the celebrated  result about "Poisson approximation to the binomial distribution", see, e.g., \cite{Feller}, \cite{Shiryayev}. It implies
\begin{equation}\label{e*1}
	\displaystyle\lim_{n\to\infty} E\varphi\left(X_n(p_n)\right)=E\varphi\left(X(\lambda)\right),\,\, \varphi\in C_b[0,\infty),
\end{equation}
where $C_b[0,\infty)$ is the space of all real valued, continuous and bounded functions on $[0,\infty)$.

Now let $x\geq 0$, $m\in {\mathbb N}$, $p_n=\dfrac{x}{n}$, $\lambda=mx$. From  (\ref{e*1}) we deduce
$$ \displaystyle\lim_{n\to\infty}E\varphi\left(X_{mn}\left(\dfrac{x}{n}\right)\right)=E\varphi\left(X(mx)\right),\,\, \varphi\in C_b[0,\infty). $$
This can be written as
\begin{equation}
	\label{e**1} 
	\displaystyle\lim_{n\to\infty}\displaystyle\sum_{k=0}^{mn}{mn\choose k}\left(\dfrac{x}{n}\right)^k\left(1-\dfrac{x}{n}\right)^{mn-k}\varphi\left(k\right)=\displaystyle\sum_{k=0}^{\infty} e^{-mx}\dfrac{(mx)^k}{k!}\varphi\left(k\right).
\end{equation}
Let $f\in C_b[0,\infty)$ and $\varphi\in C_b[0,\infty)$, $\varphi(t)=f\left(\dfrac{t}{m}\right)$, $t\geq 0$. Now (\ref{e**1}) becomes
\begin{equation*}
	\displaystyle\lim_{n\to\infty}\displaystyle\sum_{k=0}^{mn}{mn\choose k}\left(\dfrac{x}{n}\right)^k\left(1-\dfrac{x}{n}\right)^{mn-k}f\left(\dfrac{k}{m}\right)=\displaystyle\sum_{k=0}^{\infty} e^{-mx}\dfrac{(mx)^k}{k!}f\left(\dfrac{k}{m}\right).
\end{equation*}

In terms of Bernstein and Sz\'asz-Mirakyan operators this means that
\begin{equation}\label{e***1}
	\displaystyle\lim_{n\to\infty}B_{mn}\left(f(nt);\dfrac{x}{n}\right)=B_m^{[0]}\left(f(t);x\right),\,\, f\in C_b[0,\infty).
\end{equation}
This result and  similar ones involving other sequences of operators, were obtained in \cite{Adell, Adell1, Adell2, Cal_L, Cal1, Cal2} using probabilistic methods and making references to several probability distributions. Rates of convergence in relations like (\ref{e***1}) can be found in the same papers and the references therein. 

In our paper we present a general result extending (\ref{e***1}), using analytic terms and analytic methods. We apply it to many families of positive linear operators recovering some known results and establishing new ones. 
 
 In (\ref{e***1}) a specific modification of the Bernstein operators, inspired by  Poisson approximation to the binomial distribution, is presented. We will consider several other modifications corresponding to suitable families of operators. As limit operator the Sz\'asz-Mirakjan operator appears in (\ref{e***1}). It will have the same role   in several other examples. But other operators will also appear as limit operators and we will characterize some of them.
 
 Let $L_n$ and $L$ be  positive linear operators associated with random variables. In Section \ref{s2} we show that the family of the complex exponentials $t\in {\mathbb R}\to e^{ist}$, indexed by $s\in {\mathbb R}$, can be used as test family for the convergence of sequences $L_n$ toward  the operator $L$. This is presented in Theorem \ref{T2.1} which is the main result of our paper. We will apply Theorem \ref{T2.1} to several classical or new sequences of positive linear operators. Sections \ref{s3}-\ref{s9} are devoted to sequences for which the limit operator is the classical Sz\'asz-Mirakjan operator. The Gamma operator and the Weierstrass operator appear as the limit operator in Sections \ref{s10}-\ref{s11}. Other limit operators are considered in Section \ref{s12}. This section contains several characterizations of the limit operators.
 
 The necessary definitions and notations will be presented in the next sections.

\section{The complex exponentials as test functions for convergence of operators}\label{s2} Let $I\subseteq {\mathbb R}$ be an interval and $C_b(I)$ the space of all real-valued continuous and bounded functions on $I$. Let $Z^x, Z_1^x, Z_2^x,\dots$ be $I$-valued random variables whose probability distributions depend on a real parameter $x\in I$. Suppose that for each $f\in C_b(I)$ the functions $x\to Ef(Z^x)$ and $x\to Ef(Z_n^x),\, x\in I,\, n\geq 1$, are continuous on $I$. 

Consider the positive linear operators
$$ L_n:C_b(I)\to C_b(I),\,\,n\geq 1, \textrm{ and } L:C_b(I)\to C_b(I), $$
defined for $f\in C_b(I)$ and $x\in I$ by
\begin{equation}
	\label{e2.1p} L_n(f(t);x):=Ef(Z_n^x),\,\, L(f(t);x):=Ef(Z^x).
\end{equation}

 For each $s\in {\mathbb R}$ we consider the function $t\to e^{ist},\, t\in I$, and define 
$$ L_n(e^{ist};x)=L_n(\cos(st);x)+iL_n(\sin(st); x)  $$
and similarly for $L(e^{ist};x)$.
\begin{theorem}\label{T2.1}
	Suppose that for each $s\in {\mathbb R}$ and $x\in I$,
	\begin{equation}\label{e2.2p}
		\displaystyle\lim_{n\to\infty} L_n(e^{ist};x)=L(e^{ist};x),
	\end{equation}
and for each $x\in I$ the function $s\to L(e^{ist};x)$ is continuous on ${\mathbb R}$.
	Then
	\begin{equation}\label{e2.3p}
		\displaystyle\lim_{n\to\infty} L_n(f(t);x)=L(f(t);x)
	\end{equation}
for all $f\in C_b(I)$ and $x\in I$.
\end{theorem}
\begin{proof}
	Let $g_n^x(s)$ and $g^x(s)$ be the characteristic functions of $Z_n^x$ and $Z^x$, $n\geq 1$, $x\in I$. Then, according to (\ref{e2.1p})
	\begin{align*}
		& g_n^x(s):=E(e^{isZ_n^x})=L_n(e^{ist};x),\\
	&	g^x(s):=E(e^{isZ^x})=L(e^{ist};x).
	\end{align*}
Now (\ref{e2.2p}) shows that
$$ \displaystyle\lim_{n\to\infty}g_n^x(s)=g^x,\,\, s\in {\mathbb R},  $$
and by hypothesis $g^x$ is continuous on ${\mathbb R}$.

Consequently we can apply the continuity theorem of L\'evy (see \cite{Feller}, \cite{Shiryayev}) to conclude that the sequence $(Z_n^x)_{n\geq 1}$ convergences in law to $Z^x$. Furthermore, this implies (see \cite{Feller}, \cite{Shiryayev})
\begin{equation}
	\label{e2.4p} \displaystyle \lim_{n\to\infty} Ef(Z_n^x)=Ef(Z^x),
\end{equation}
and now (\ref{e2.3p}) is a consequence of (\ref{e2.1p}) and (\ref{e2.4p}).
\end{proof}

In the next sections  we will be concerned with positive linear operators $L_n$ and $L$ for which it is easy to identify the corresponding random variables. See also \cite{Adell, Adell1, Adell2, Cal_L, Cal1, Cal2}.

\section{The Baskakov type operators ${ B}_{n}^{[c]}$}\label{s3}
Let $c \in \mathbb{R}$, $n \in \mathbb{R}$, $n > c$ for $c\geq0$ and $-n/c \in \mathbb{N}$ for $c<0$. Furthermore let  $I_c = [0,\infty)$ for $c\geq0$ and $I_c=[0,-1/c]$ for $c < 0$. Consider $f:I_c \longrightarrow \mathbb{R}$  given in such a way that the corresponding integrals and series are convergent.

The Baskakov-type operators are defined as follows (see \cite{Bas_VB, RIMA_RaHe, India})
$$
{ B}_{n}^{[c]} (f;x) = \sum_{j=0}^{\infty} p_{n,j}^{[c]}(x) f \left ( \frac{j}{n} \right ) ,
$$
where 
\begin{equation}
	\label{eq1}
	p_{n,j}^{[c]}(x)=
	\left \{ \begin{array}{cl}
		\displaystyle \frac{n^j}{j!} x^j  e^{-nx} &, \, c = 0 ,\\
		\displaystyle  \frac{n^{c,\overline{j}}}{j!} x^j (1+cx)^{-\left ( \frac{n}{c}+j\right)} &, \, c \not= 0,
	\end{array} \right .
\end{equation}
and $
a^{c,\overline{j}} := \prod_{l=0}^{j-1}  (a+cl),\quad 
a^{c,\overline{0}} :=1.
$ 

For $c=-1$ one recovers the well known Bernstein operators
\begin{equation*}
	B_n^{[-1]}(f;x):=\sum^{n}_{j=0}f\left(\frac{j}{n}\right)p_{n,j}^{[-1]}(x),\textrm{ where } p_{n,j}^{[-1]}(x):=\binom{n}{j}x^j(1-x)^{n-j},\,\, x\in [0,1].
\end{equation*}

 The classical Baskakov operators are obtained for $c=1$ and are defined as follows
\begin{equation*}
	B_n^{[1]}(f;x):=\sum^{\infty}_{j=0}f\left(\frac{j}{n}\right)p_{n,j}^{[1]}(x),\textrm{ where } p_{n,j}^{[1]}(x):=\binom{n+j-1}{j}\frac{x^j}{(1+x)^{n+j}},\,\, x\in [0,\infty).
\end{equation*}

The classical
Sz\'{a}sz-Mirakjan operators are Baskakov type operators with $c=0$ defined by 
\begin{equation}\label{smeq1}
	B_n^{[0]}(f;x):=\sum^{\infty}_{j=0}p_{n,j}^{[0]}(x)f\left(\frac{j}{n}\right), \textrm{ where } p_{n,j}^{[0]}(x)=e^{-nx}\frac{(nx)^j}{j!},\,\, x\in [0,\infty).
\end{equation}


\begin{theorem}
	Let $m\in {\mathbb N}$, $c\in {\mathbb R}$ and $x\in I_c$ be given. Then
\begin{equation}\label{e3.1} \displaystyle\lim_{n\to\infty} { B}_{mn}^{[c]}\left(f(nt);\dfrac{x}{n}\right)= B_{m}^{[0]}(f(t);x),\,\, f\in C_b(I_c).  \end{equation}
\end{theorem}
\begin{proof}
	Consider first the case $c=0$. Using elementary calculations we find 
	$$  B_{mn}^{[0]}\left(f(nt);\dfrac{x}{n}\right)=B_{m}^{[0]}\left(f(t);x\right).$$
	Now let $c\ne 0$. Then
	\begin{align*}
		{ B}_{n}^{[c]}\left(e^{ist};x\right)&=\displaystyle\sum_{j=0}^{\infty}\dfrac{n^{c,\overline{j}}}{j!}x^j(1+cx)^{-n/c-j}e^{isj/n}\\
		&=(1+cx)^{-n/c}\displaystyle\sum_{j=0}^{\infty}\dfrac{n(n+c)\dots(n+(j-1)c)}{j!}\left(\dfrac{xe^{is/n}}{1+cx}\right)^j\\
		&=(1+cx)^{-n/c}\displaystyle\sum_{j=0}^{\infty}{n/c+j-1\choose j}\left(\dfrac{cxe^{is/n}}{1+cx}\right)^j\\
		&=(1+cx)^{-n/c}\left(1-\dfrac{cxe^{is/n}}{1+cx}\right)^{-n/c},
	\end{align*}
and so
\begin{equation}\label{e3.5}
{	B}_{n}^{[c]}(e^{ist};x)=\left(1-cx(e^{is/n}-1)\right)^{-n/c}.
\end{equation}
Define
\begin{align}
	& L_n(f(t);x):={ B}_{mn}^{[c]}\left(f(nt);\dfrac{x}{n}\right),\label{e3.6}\\
	& L(f(t);x):=B_{m}^{[0]}(f(t);x).\label{e3.7}
	\end{align}
From (\ref{e3.5}), (\ref{e3.6}), (\ref{e3.7}) we obtain
\begin{align*}
	\displaystyle\lim_{n\to\infty} L_n(e^{ist};x)&=\lim_{n\to\infty}{ B}_{mn}^{[c]} \left(e^{isnt};\dfrac{x}{n}\right)\\
	&=\displaystyle\lim_{n\to\infty}\left(1-c\dfrac{x}{n}\left(e^{is/m}-1\right)\right)^{-mn/c}=e^{mx\left(e^{is/m}-1\right)}\\
	&=B_{m}^{[0]}\left(e^{ist};x\right)=L(e^{ist};x).
\end{align*}
Thus (\ref{e2.2p}) is satisfied and (\ref{e3.1}) is a consequence of (\ref{e2.3p}), (\ref{e3.6}) and (\ref{e3.7}).
\end{proof}
	\section{The $k$th order Kantorovich
	modification of the Baskakov type operators}\label{s4}
The $k$-th order Kantorovich modifications of the operators ${ B}_{n}^{[c]}$ are defined by
$$ {B}_{n}^{[c](k)}:=D^k\circ { B}_{n}^{[c]}\circ { I}_k, \,\, k\in {\mathbb N},$$
where $D^{k}$ denotes the $k$-th order ordinary differential operator and
$${ I}_kf=f,\textrm{ if } k=0 \textrm{ and } \left({ I}_kf\right)(x)=\displaystyle\int_0^x\dfrac{(x-t)^{k-1}}{(k-1)!}f(t) dt, \textrm{ if } k\in \mathbb{N}.   $$

We have ${ B}_{n}^{[c](k)}\left(f(t);x)\right):=\dfrac{d^k}{dx^k}{ B}_{n}^{[c]}\left((I_kf)(t);x\right).$
Using (\ref{e3.5}) we get
\begin{align}
{ B}_{n}^{[c](k)}(e^{ist};x)&=\dfrac{1}{(is)^k}\dfrac{d^k}{dx^k}{ B}_{n}^{[c]}\left(e^{ist};x\right)\nonumber\\
	&=n^{c,\overline{k}}\left(\dfrac{e^{is/n}-1}{is}\right)^k\left(1-cx(e^{is/n}-1)\right)^{-n/c-k},\,\, c\ne 0,\label{e4.1}\\
		{ B}_{n}^{[c](k)}(e^{ist};x)&=
		\left(\dfrac{n(e^{is/n}-1)}{is}\right)^ke^{nx(e^{is/n}-1)},\,\, c=0.\label{e4.2}
\end{align}
The image of the constant function {\bf 1} under $	{B}_{n}^{[c](k)}$ is obtained with $s=0$ and is the constant function $\dfrac{n^{c,\overline{k}}}{n^k}{\bf 1}$. Therefore the operators 
	\begin{equation}\label{e4.3}
		{V}_{n}^{[c](k)}:=(n^k/n^{c,\overline{k}})	{ B}_{n}^{[c](k)}
	\end{equation}
preserve the constant function {\bf 1}.

\begin{theorem}\label{T4.1}
	For $m\in {\mathbb N}$, $c\in {\mathbb R}$, $x\in I_c$, and $f\in C_b(I)$, we have
	\begin{equation}\label{e4.4}
	\displaystyle\lim_{n\to\infty}	V_{mn}^{[c](k)}\left(f(nt);\dfrac{x}{n}\right):=V_{m}^{[0](k)}\left(f(t);x\right).
	\end{equation}
\end{theorem}
\begin{proof}
	According to Theorem \ref{T2.1}, it suffices to show
that (\ref{e4.4}) holds for $f(t):=e^{ist},\, s\in{\mathbb R}$.

Let $c\neq 0$. Using (\ref{e4.3}) and (\ref{e4.1}) we have 
\begin{align*}
	\displaystyle\lim_{n\to\infty} V_{mn}^{[c](k)}\left(e^{isnt};\dfrac{x}{n}\right)&=\displaystyle\lim_{n\to\infty}\left(\dfrac{m}{is}(e^{is/m}-1)\right)^k\left(1-c\dfrac{x}{n}(e^{is/m}-1)\right)^{-mn/c-k}\\
	&=\left(\dfrac{m}{is}(e^{is/m}-1)\right)^ke^{mx(e^{is/m}-1)}\\
	&=V_{m}^{[0](k)}(f(t);x).
\end{align*}
Using (\ref{e4.3}) and (\ref{e4.2}) it is easy to check that (\ref{e4.4}) holds for $f(t)=e^{ist},\,s\in{\mathbb R}$, also for $c=0$. Now (\ref{e4.4}) for $f\in C_b(I)$ is a consequence of Theorem \ref{T2.1}.
\end{proof}
\begin{remark}
	Since $\displaystyle\lim_{n\to\infty} n^k/n^{c,\overline{k}}=1$, from (\ref{e4.3}) and (\ref{e4.4}) we derive
	$$ \displaystyle\lim_{n\to\infty} { B}_{mn}^{[c](k)}\left(f(nt);\dfrac{x}{n}\right)={ B}_{m}^{[c](k)}\left(f(t);x\right).  $$
\end{remark}
\section{Discrete operators associated with $	V_{n,\rho}^{[c](k)}$}\label{s5}
Depending on a parameter  $\rho \in \mathbb{R}^+$ the linking operators are given by
\begin{equation*}
	({ B}_{n,\rho}^{[c]} f)(x) = \sum_{j=0}^{\infty} F_{n,j}^{[c],\rho} (f) p_{n,j}^{[c]} (x),\, f:I_c \longrightarrow \mathbb{R},
\end{equation*}
where 
$$
F_{n,j}^{[c],\rho} (f) = \left \{
\begin{array}{cll}
	f(0) & , & j=0, \, c \in \mathbb{R},
	\\
	\displaystyle f \left (-\frac{1}{c} \right ) & , & j=-\frac{n}{c} , \, c <0,
	\\
	\displaystyle \int_{I_c} \mu_{n,j}^{[c],\rho} (t) f(t) dt & , & \mbox{otherwise},
\end{array} \right .
$$
with
$$
\mu_{n,j}^{[c],\rho} (t) = \left \{ \begin{array}{cl}
	\displaystyle \frac{(-c)^{j\rho}}{B \left (j\rho,- \left (\frac{n}{c}+j \right )\rho \right )} t^{j\rho -1} (1+ct)^{-\left ( \frac{n}{c}+j\right)\rho -1} &, \, c < 0 , \\
	\displaystyle \frac{(n\rho)^{j\rho}}{\Gamma (j \rho )} t^{j\rho-1}  e^{-n\rho t} &, \, c = 0  ,\\
	\displaystyle \frac{c^{j\rho}}{B \left (j\rho,\frac{n}{c}\rho+1 \right )} t^{j\rho -1} (1+ct)^{-\left ( \frac{n}{c}+j\right)\rho -1} &, \, c > 0 .
\end{array} \right .
$$
By $B(x,y) = \int_0^1 t^{x-1} (1-t)^{y-1} dt$, $x,y > 0$ we denote Euler's Beta function and by $\Gamma$ the Gamma function.

These operators represent a link between Baskakov type operators ${ B}_{n}^{[c]}$ (for $\rho=\infty$) and the genuine Baskakov-Durrmeyer type operators (for $\rho=1$). Convergence results concerning the linking operators can be found in \cite[Section 5]{K2}, \cite[Th. 2.3]{K4}, \cite[Th. 1, Th. 2, Cor. 1]{K7} and \cite[Th. 4]{K12}.

Let
$
{ B}_{n,\rho}^{[c](k)}:=D^k \circ { B}_{n,\rho}^{[c]}\circ I_k
$ be the $k$-th order Kantorovich modification of the operators ${ B}_{n,\rho}^{[c]}$ and $
	V_{n,\rho}^{[c] (k)} := \dfrac{(n\rho)^{c,\underline{k}}}{\rho^k n^{c,\overline{k}}} { B}_{n,\rho}^{[c] (k)}.
$
For $\rho\in {\mathbb N}$ explicit representation are given in \cite{Ha}, \cite{Hb} and \cite{India}.

The discrete operator  associated with $V_{n,\rho}^{[c](k)}$ was introduced in \cite{He-Na-Ra} as follows

\begin{equation*}
	D_{n,\rho}^{[c](k)}\left(f(t),x\right)=\displaystyle\sum_{j=0}^{\infty}f\left(\dfrac{(2j+k)\rho+k}{2(n\rho-kc)}\right)p_{n+kc,j}^{[c]}(x).
\end{equation*}
Note that $D_{n,\rho}^{[c](k)}e_0=e_0$.
For $c\ne 0$ and $\rho n>kc$, we have
\begin{align*}D_{n,\rho}^{[c](k)}\left(e^{ist};x\right)&=e^{\frac{isk(\rho+1)}{2(n\rho-kc)}}(1+cx)^{-n/c-k}\sum_{j=0}^{\infty}{n/c+k+j-1\choose j}\left(\dfrac{cxe^{is\rho/(n\rho-kc)}}{1+cx}\right)^j\\
&=e^{\frac{isk(\rho+1)}{2(n\rho-kc)}}(1+cx)^{-n/c-k}\left(1-\dfrac{cxe^{is\rho/(n\rho-kc)}}{1+cx}\right)^{-n/c-k}. \end{align*}
Therefore,
\begin{equation}
	\label{e5.1}
	D_{n,\rho}^{[c](k)}\left(e^{ist};x\right)=e^{\frac{isk(\rho+1)}{2(n\rho-kc)}}\left(1-cx(e^{is\rho/(n\rho-kc)}-1)\right)^{-n/c-k}.
\end{equation}
Similarly,
\begin{equation}
	\label{e5.2} D_{n,\rho}^{[0](k)}(e^{ist};x)=e^{\frac{isk(\rho+1)}{2n\rho}+nx(e^{is/n}-1).}
\end{equation}
\begin{theorem}
	Let $m\in {\mathbb N}$, $c\in {\mathbb R}$,  $x\in I_c$, $f\in C_b(I)$. Then
	\begin{equation}\label{e5.3}
		\displaystyle\lim_{n\to\infty}D_{mn,\rho}^{[c](k)}\left(f(nt);\dfrac{x}{n}\right)=D_{m,\rho}^{[0](k)}\left(f(t);x\right).
		\end{equation}
\end{theorem}
\begin{proof} Using (\ref{e5.1}) and (\ref{e5.2}) it is easy to check that (\ref{e5.3}) is valid for $f(t)=e^{ist},\,s\in {\mathbb R}$. Therefore, (\ref{e5.3}) is a consequence of Theorem \ref{T2.1}.
 	\end{proof}
 
 \section{Other examples}\label{s6}
 \subsection{} 
 For a non-negative constant $a$, independent of $n$, in \cite{111} it was introduced the operator
 $$ Q_n^a(f;x)=\displaystyle\sum_{k=0}^{\infty}q_{n,k}^a(x)f\left(\dfrac{k}{n}\right),   $$
 where $$q_{n,k}^a(x)=e^{-\frac{ax}{1+x}}\frac{\sum_{i=0}^k{k\choose i}n^{1,\overline{i}}a^{k-i}}{k!}\dfrac{x^k}{(1+x)^{n+k}}.$$
 For the operators $Q_n^a$  we have (see \cite[1.10.2]{Gupta_Rassias})
 $$ Q_n^a(e^{ist};x)=(1+x-xe^{is/n})^{-n}\exp\left(\dfrac{ax}{1+x}\left(e^{is/n}-1\right)\right)  $$
 Consequently,
 \begin{align*}\displaystyle\lim_{n\to\infty}Q_{mn}^a\left(e^{isnt};\dfrac{x}{n}\right)&=\lim_{n\to\infty}\left(1+x/n\left(1-e^{is/m}\right)\right)^{-mn}e^{\frac{ax}{n+x}\left(e^{is/m}-1\right)}\\
 	&=e^{mx(e^{is/m}-1)}=B_{m}^{[0]}(e^{ist};x).
 \end{align*}
Using Theorem \ref{T2.1} we get for $f\in C_b[0,\infty)$, $m\geq 1$, $x\in [0,\infty)$,
$$ \displaystyle\lim_{n\to\infty} Q_{mn}^a\left(f(nt);\dfrac{x}{n}\right)=B_{m}^{[0]}(f(t);x).  $$

\subsection{} 
Let $c\ne 0$ and 
\begin{align*}
	\sigma:I_c\to I_{-c},\,\, \sigma(x)=\dfrac{x}{1+cx},\\
	\psi:I_c\to I_{-c},\,\, \psi(x)=\dfrac{x}{1-cx}.
	\end{align*}
Let $B_{n,1}^{[c](k)}$ be the
kth order Kantorovich modifications of the genuine Baskakov-Durrmeyer type operators.
In \cite{Maja2015} the second author of this paper presented a method to relate the
 kth order Kantorovich modifications of Durrmeyer type variants of
 Bleimann-Butzer-Hahn operators for $c<0$ and
Meyer-K\"onig and Zeller operators  for $c>0$ with $B_{n,1}^{[c](k)}$ as follows
 
$$\overline{B}_{n,1}^{[c](k)}(f;x):=B_{n,1}^{[c](k)}(f\circ\sigma;\psi(x)).$$
For details see \cite[Sec.6]{Maja2015}.
Using Theorem \ref{T2.1} we derive
\begin{align}
	\displaystyle\lim_{n\to\infty}\overline{B}_{mn,1}^{[c](k)}\left(f(nt);\dfrac{x}{n}\right)&=\displaystyle\lim_{n\to\infty}B_{mn,1}^{[c](k)}\left((f\circ\sigma)(nt);\dfrac{\psi(x)}{n}\right)\nonumber\\
	&=B_{m,1}^{[0](k)}\left((f\circ\sigma)(t);\psi(x)\right).\label{e.Maja}
\end{align}

\section{The Bleimann-Butzer-Hahn operators}\label{s7}
Let  \begin{equation}\label{BBH}H_n(f(t);x):=(1+x)^{-n}\displaystyle\sum_{k=0}^n{n\choose k}x^kf\left(\dfrac{k}{n-k+1}\right)\end{equation} be the Bleimann-Butzer-Hahn operators \cite{BBH}.  Using the probabilistic approach it was proved in \cite[Th. 2 a)]{Cal_L} that
\begin{equation}
	\label{e6.1} \displaystyle\lim_{n\to\infty}H_{mn}\left(f\left(\dfrac{nt}{1+t}\right);\dfrac{x}{n}\right)=B_m^{[0]}\left(f(t);x\right).
\end{equation}
Here is an analytic  proof based on Theorem \ref{T2.1}. In fact, all that we have to prove is that (\ref{e6.1}) holds for $f(t):=e^{ist}$, $s\in {\mathbb R}$.

This immediately follows from 
\begin{align*}
	\displaystyle\lim_{n\to\infty}H_{mn}\left(e^{isnt/(1+t)};\dfrac{x}{n}\right)&=\displaystyle\lim_{n\to\infty}\left(1+\dfrac{x}{n}\right)^{-mn}\sum_{k=0}^{mn}{mn\choose k}\left(\dfrac{x}{n}\right)^ke^{isnk/(mn+1)}\\
	&=\displaystyle\lim_{n\to\infty}\left(1+\dfrac{x}{n}\right)^{-mn}\left(1+\dfrac{x}{n}e^{isn/(mn+1)}\right)^{mn}\\
	&=e^{mx(e^{is/m}-1)}=B_m^{[0]}(e^{ist};x).
\end{align*}

\begin{exam} It was proved in \cite[Corollary 3]{Adell} that
	\begin{equation}
		\label{e9.8} \displaystyle\lim_{n\to\infty} H_{mn}\left[ f\left(\dfrac{(mn+1)t}{m(1+t)}\right);\dfrac{x}{n}\right]=B_m^{[0]}(f(t);x).
	\end{equation}
	Here is the proof based on Theorem \ref{T2.1}. We have to prove (\ref{e9.8}) for $f(t)=e^{ist},\,\, s\in {\mathbb R}$. 
	\begin{align*}
		&	\displaystyle\lim_{n\to\infty}H_{mn}\left[\exp\left(is\dfrac{(mn+1)t}{m(1+t)}\right);\dfrac{x}{n}\right]\\
		&=\displaystyle\lim_{n\to\infty}\left(1+\dfrac{x}{n}\right)^{-mn}\sum_{k=0}^{mn}{mn\choose k}\left(\dfrac{x}{n}\right)^k\exp\left(is\dfrac{(mn+1)\frac{k}{mn-k+1}}{m\left(1+\frac{k}{mn-k+1}\right)}\right)\\
		&=\displaystyle\lim_{n\to \infty}\left(1+\dfrac{x}{n}\right)^{-mn}\sum_{k=0}^{mn}{mn\choose k}\left(\dfrac{x}{n}\right)^k\left(e^{is/m}\right)^k\\
		&=\displaystyle\lim_{n\to\infty}\left(1+\dfrac{x}{n}\right)^{-mn}\left(1+\dfrac{x}{n}e^{is/m}\right)^{mn}=e^{-mn}e^{mxe^{is/m}}\\
		&=e^{mx(e^{is/m}-1)}=B_m^{[0]}(e^{ist};x).
	\end{align*}
	This ends the proof.
\end{exam}
Concerning this example see also \cite[Remark 2, p.496]{Adell}, where it is proved that

$$\left|H_{mn}\left( f\left(\dfrac{(mn+1)t}{m(1+t)}\right);\dfrac{x}{n}\right)-H_{mn}\left(f\left(\dfrac{nt}{1+t}\right);\dfrac{x}{n}\right)\right| \leq 2\omega\left(f;\dfrac{x}{mn}\right), $$
$\omega$ being the first modulus of continuity.

\section{Meyer-K\"onig and Zeller operators}
\label{s8}
The operators of	Meyer-K\"onig and Zeller \cite{MKZ} in the slight modification of Cheney and Sharma \cite{Sharma} are defined for $ f\in C[0,1]$ as follows
\begin{equation}\label{MKZ}{ M}_n(f;x)=\left\{\begin{array}{l}\displaystyle\sum_{k=0}^{\infty}{n+k\choose k}x^k(1-x)^{n+1}f\left(\dfrac{k}{n+k}\right),\,\,x\in[0,1),\\
		\vspace{-0.3cm}\\
		f(1), x=1.
	\end{array} \right. \end{equation}
\subsection{Modified Meyer-K\"onig and Zeller operators converging to Sz\'asz-Mirakyan operator}


Let us remark that
$$ M_n\left(f\left(\dfrac{t}{1-t}\right);x\right)=(1-x)^{n+1}\sum_{k=0}^{\infty}f\left(\dfrac{k}{n}\right){n+k\choose k}x^k, $$
and consequently. In particular 
\begin{align*}
	M_{mn}\left(f\left(\frac{nt}{1-t}\right);\frac{x}{n}\right)&=\left(1-\dfrac{x}{n}\right)^{mn+1}\sum_{k=0}^{\infty}{mn+k\choose k}\left(\dfrac{x}{n}\right)^kf\left(\frac{n\frac{k}{mn+k}}{1-\frac{k}{mn+k}}\right)\\
	&=\left(1-\dfrac{x}{n}\right)^{mn+1}\sum_{k=0}^{\infty}{mn+k\choose k}\left(\dfrac{x}{n}\right)^kf\left(\dfrac{k}{m}\right),
\end{align*}
and so,
\begin{align*}
	\displaystyle\lim_{n\to\infty}	M_{mn}\left(e^{is\dfrac{nt}{1-t}};\dfrac{x}{n}\right)&=\displaystyle\lim_{n\to\infty}\left(1-\dfrac{x}{n}\right)^{mn+1}\sum_{k=0}^{\infty}{mn+k\choose k}\left(\dfrac{x}{n}\right)^k\left(e^{is/m}\right)^k\\
	&=\displaystyle\lim_{n\to\infty}\left(1-\dfrac{x}{n}\right)^{mn+1}\dfrac{1}{(1-\frac{x}{n}e^{is/m})^{mn}}= e^{-mx}e^{mxe^{is/m}}.
\end{align*}
We see that
$$\displaystyle\lim_{n\to\infty}M_{mn}\left(e^{is\frac{nt}{1-t}};\dfrac{x}{n}\right)= e^{mx(e^{is/m}-1)}=B_m^{[0]}(e^{ist};x),  $$
and therefore,
$$\displaystyle\lim_{n\to\infty}M_{mn}\left(f\left(\frac{nt}{1-t}\right);\dfrac{x}{n}\right)= B_m^{[0]}\left(f(t);x\right).  $$
In a similar way it can be proved that

$$ \displaystyle\lim_{n\to\infty}M_{mn}\left[f\left(\dfrac{(mn+1)t}{m(1-t)}\right);\dfrac{x}{n}\right]=B_m^{[0]}(f(t);x). $$

\subsection{Modified Meyer-K\"onig and Zeller operators converging to Gamma operator}

Here we consider another modification of Meyer-K\"onig and Zeller operators, namely
\begin{align*}
	M_m\left[f\left(\dfrac{t}{n(1-t)}\right);\dfrac{nx}{1+nx}\right]&=\left(1-\dfrac{nx}{1+nx}\right)^{m+1}\displaystyle\sum_{k=0}^{\infty}f\left(\frac{\frac{k}{m+k}}{n\left(1-\frac{k}{m+k}\right)}\right){m+k\choose k}\left(\dfrac{nx}{1+nx}\right)^k\\
	&=\dfrac{1}{(1+nx)^{m+1}}\displaystyle\sum_{k=0}^{\infty} f\left(\dfrac{k}{mn}\right){m+k\choose k}\left(\dfrac{nx}{1+nx}\right)^k.
\end{align*}
In particular,
\begin{align*}
	\displaystyle\lim_{n\to\infty}	M_m\left[e^{is\frac{t}{n(1-t)}};\dfrac{nx}{1+nx}\right]&=\displaystyle\lim_{n\to\infty}
	\dfrac{1}{(1+nx)^{m+1}}\sum_{k=0}^{\infty}{m+k\choose k}\left(\dfrac{nx}{1+nx}e^{is/mn}\right)^k\\&=\displaystyle\lim_{n\to\infty}\dfrac{1}{(1+nx)^{m+1}}\dfrac{1}{\left(1-\frac{nx}{1+nx}e^{is/mn}\right)^{m+1}}\\
	&=\displaystyle\lim_{n\to\infty}\dfrac{1}{[1+nx(1-e^{is/mn})]^{m+1}}=\dfrac{1}{\left(1-\frac{isx}{m}\right)^{m+1}}.
\end{align*}
Using (\ref{e7.2})
we get
\begin{equation}\label{e.Gamma1}\displaystyle\lim_{n\to\infty}M_m\left[f\left(\dfrac{t}{n(1-t)}\right);\dfrac{nx}{1+nx}\right]=G_{m+1}\left( f(t);\dfrac{m+1}{m}x\right).\end{equation}

\section{Convergence of Lototsky-Schnabl operators}\label{s9}
Let $\lambda:[0,1]\to[0,1]$ be continuous, $\lambda(0)=1$. For $f\in C[0,1]$, $a\in[0,1]$, $x\in[0,1]$, we consider the function
$$ f_{x,a}(t):=f(at+(1-a)x), \,\, t\in[0,1]. $$ 
Using the Bernstein operators $(B_r^{[-1]})_{r\geq 1}$ one defines the Lototsky-Schnabl operators $A_{n,\lambda}:C[0,1]\to C[0,1]$,
$$ A_{n,\lambda}\left(f(t);x\right):=\displaystyle\sum_{r=0}^n {n\choose r}\lambda(x)^r (1-\lambda(x))^{n-r}B_r^{[-1]}\left(f_{x,\frac{r}{n}}(t);x\right).  $$
For details see \cite{Adell2}.
\begin{theorem} Given $f\in C_b[0,\infty)$, $x\in[0,\infty)$, $m\geq 1$, we have
	\begin{equation}
		\label{eA9.1}
		\displaystyle\lim_{n\to\infty} A_{mn,\lambda}\left(f(nt);\dfrac{x}{n}\right)=B_m^{[0]}\left(f(t);x\right),
	\end{equation}
	where $B_m^{[0]}$ is the classical Sz\'asz-Mirakjan operator.
\end{theorem}
\begin{proof}
	Let $L_n(f(t);x):=A_{mn,\lambda}\left(f(nt);\dfrac{x}{n}\right)$. Remark that $$f_{\frac{x}{n},\frac{r}{mn}}(nt)=f\left(\dfrac{rt}{m}+\left(1-\dfrac{r}{mn}\right)\dfrac{x}{n}\right)$$ and 
	$$ B_r^{[-1]}\left(f_{\frac{x}{n},\frac{r}{mn}}(nt);\dfrac{x}{n}\right)=\displaystyle\sum_{k=0}^r{r\choose k}\left(\dfrac{x}{n}\right)^k\left(1-\dfrac{x}{n}\right)^{r-k}f\left(\dfrac{k}{m}+\left(1-\dfrac{r}{mn}\right)\dfrac{x}{n}\right). $$
	With $f(t):=e^{ist}$ we have
	$$ B_r^{[-1]}\left(f_{\frac{x}{n},\frac{r}{mn}}(nt);\dfrac{x}{n}\right)=\left(1-\dfrac{x}{n}+\dfrac{x}{n}e^{is/m}\right)^r e^{is\left(1-\frac{r}{mn}\right)\frac{x}{n}}.$$
	Therefore,
	\begin{align*}
		L_n(e^{ist};x)&=\displaystyle\sum_{r=0}^{mn}{mn\choose r}\lambda\left(\dfrac{x}{n}\right)^r\left(1-\lambda\left(\dfrac{x}{n}\right)\right)^{mn-r}\left(1-\dfrac{x}{n}+\dfrac{x}{n}e^{is/m}\right)^re^{is\left(1-\frac{r}{mn}\right)\frac{x}{n}}\\
		&=e^{isx/n}\left[1-\lambda\left(\dfrac{x}{n}\right)+\lambda\left(\dfrac{x}{n}\right)\left(1+\dfrac{x}{n}\left(e^{is/m}-1\right)\right)e^{-isx/mn^2}\right]^{mn}.
	\end{align*}
	By straightforward calculations we get
	$$\displaystyle\lim_{n\to\infty}L_n\left(e^{ist};x\right)=e^{mx(e^{is/m-1})}=B_m^{[0]}\left(e^{ist};x\right).  $$
	Now Theorem \ref{T2.1} shows that
	$$ \displaystyle\lim_{n\to\infty} A_{mn,\lambda}\left(f(nt);\dfrac{x}{n{}}\right)=\lim_{n\to\infty} L_n\left(f(t);x\right)=B_m^{[0]}(f(t);x), $$
	and this concludes the proof.
\end{proof}
\begin{remark}
	An estimate of the rate of convergence in (\ref{eA9.1}) can be found in \cite[Theorem 4]{Adell2}.
\end{remark}

\section{Convergence toward Gamma type operators}\label{s10}
Consider the Gamma operator
\begin{equation}
	\label{e7.1} G_m(f(t);x):=\dfrac{x^{-m}}{(m-1)!}\int_0^{\infty} f\left(\dfrac{t}{m}\right)t^{m-1}e^{-t/x}du,
\end{equation}
where $f\in C_b[0,\infty)$, $x\in[0,\infty)$. (If $x=0$, $G_m(f(t);0)=f(0))$.

It is not difficult to prove  that
\begin{equation}
	\label{e7.2}
	G_m\left(e^{ist};x\right)=\left(1-\dfrac{isx}{m}\right)^{-m},\,\, s\in{\mathbb R}
\end{equation}

As in Section \ref{s3}, it can be proved that
$$ V_{m}^{[c](k)}(e^{ist};x)=\left(\dfrac{m}{is}(e^{is/m}-1)\right)^k\left(1-cx(e^{is/m}-1)\right)^{-\frac{m}{c}-k},  $$
and so, for $c>0$,
\begin{align*}
	\displaystyle\lim_{n\to\infty} V_{m}^{[c](k)}\left(e^{ist/n};nx\right)&=\lim_{n\to\infty}\left(\dfrac{mn}{is}\left(e^{\frac{is}{mn}}-1\right)\right)^k\left(1-cnx(e^{is/mn}-1)\right)^{-\frac{m}{c}-k}\\
	&=\left(1-\dfrac{iscx}{m}\right)^{-\frac{m}{c}-k}=G_{\frac{m}{c}+k}\left(e^{ist};\dfrac{m+ck}{m}x\right).
\end{align*}
Now from Theorem \ref{T2.1} we deduce
\begin{theorem}
	With the above notation we have
	\begin{equation}\label{e.Gamma}
		\displaystyle\lim_{n\to\infty} V_{m}^{[c](k)}\left(f\left(\dfrac{t}{n}\right);nx\right)=G_{\frac{m}{c}+k}\left(f(t);\dfrac{m+ck}{m}x\right). \end{equation}
\end{theorem}
For $k=0$ and $c=1$, this is Theorem 2 d) in \cite{Cal_L}.

\section{Convergence toward Weierstrass operator}\label{s11}
Let $a>0$. The Weierstrass  operator is defined by (see \cite{FA-MC})
$$ W_a\left(f(t);x\right):=\dfrac{1}{\sqrt{2\pi a}}\int_{\mathbb R} f(t)e^{-\frac{(t-x)^2}{2a}}dt,\,\, f\in C_b[0,\infty), x\geq 0.  $$
The following results were briefly presented in \cite[(24.25)]{Vladislav} and \cite{Popa_Rasa}. Here we give detailed proofs based on Theorem \ref{T2.1}.
\begin{theorem}
	\label{T8.1} Let $0<u<1$ be given, and let $b_n>0$ such that $\lim_{n\to\infty}nb_n^2=a$. Then
	\begin{equation}
		\label{e8.1} 
		\displaystyle\lim_{n\to\infty} B_n^{[-1]}\left(f\left(x+\dfrac{nb_n}{\sqrt{u(1-u)}}(t-u)\right);u\right)=W_a\left(f(t);x\right),
	\end{equation}
where $B_n^{[-1]}$, $n\geq 1$, are the classical Bernstein operators.
\end{theorem}
\begin{proof}
	Let \begin{equation}\label{e8.2}L_n\left(f(t);x\right):=B_n^{[-1]}\left(f\left(x+\dfrac{nb_n}{\sqrt{u(1-u)}}(t-u)\right);u\right).\end{equation}
	Then, by a straightforward calculation we get
	\begin{align*}
		L_n\left(e^{ist};x\right)&:=B_n^{[-1]}\left(e^{ is\left(x+\frac{nb_n}{\sqrt{u(1-u)}}(t-u)\right)};u\right)\\
		&=e^{isx}\sum_{j=0}^n{n\choose j}u^j(1-u)^{n-j}e^{is\frac{nb_n}{\sqrt{u(1-u)}}\left(\frac{j}{n}-u\right)}\\
		&=e^{ist}\left[(1-u)e^{\frac{-isub_n}{\sqrt{u(1-u)}}}+ue^{\frac{is(1-u)b_n}{\sqrt{u(1-u)}}}\right]n.
	\end{align*}
Remark that
$$ \displaystyle\lim_{n\to \infty} L_n\left(e^{ist};x\right)=e^{isx}e^{\lim_{n\to\infty}\left(\frac{(1-u)e^{\frac{-isub_n}{\sqrt{u(1-u)}}}+ue^{\frac{is(1-u)b_n}{\sqrt{u(1-u)}}}}{b_n^2}nb_n^2\right)}.  $$
Since $nb_n^2\to a$ we have $b_n\to 0$ and it readily follows that
\begin{equation}\label{e8.3}
	\displaystyle\lim_{n\to\infty} L_n(e^{ist};x)=e^{isx-\frac{1}{2}as^2}.
\end{equation}
On the other hand, using the well-known characteristic function of a Gaussian random variable with mean $x$ and variance $a$, we can write
\begin{equation}
	\label{e8.4} 
	W_a(e^{ist};x)=\dfrac{1}{\sqrt{2\pi a}}\int_{{\mathbb R}} e^{ist} e^{-\frac{(t-x)^2}{2a}}dt=e^{isx-\frac{1}{2}as^2}.
\end{equation}
From (\ref{e8.3}), (\ref{e8.4}) and Theorem \ref{T2.1} we infer that
\begin{equation}
	\label{e8.5} \displaystyle\lim_{n\to\infty} L_n(f(t);x)=W_a(f(t);x).
\end{equation}
Now (\ref{e8.1}) is a consequence of (\ref{e8.2}) and (\ref{e8.5}).
\end{proof}
We consider again a sequence $b_n>0$, $n\geq 1$, such that $\displaystyle\lim_{n\to\infty}nb_n^2=a>0$. The following Bernstein-Schnabl type operators were introduced in \cite{Rasa_Bari} and investigated in \cite{Ivan1}, \cite{Ivan2}, \cite{Vladislav}, \cite{Popa_Rasa}, \cite{AGR}, \cite{RACSAM}:
$$ S_{n,b_n}\left(f(t);x\right):=\dfrac{1}{(2nb_n)^n}\int_{x-nb_n}^{x+nb_n}\cdots \int_{x-nb_n}^{x+nb_n}f\left(\dfrac{u_1+\dots+u_n}{n}\right)du_1\cdots du_n.  $$
\begin{theorem}
	Let $ f\in C_b[0,\infty)$ and $x\geq 0$. Then
	\begin{equation}
		\label{e8.6}
		\displaystyle\lim_{n\to\infty} S_{n,b_n}(f(t);x)=W_{a/3}(f(t);x).
	\end{equation}
\end{theorem}
\begin{proof}
	Using Theorem \ref{T2.1} it suffices to show that
	\begin{equation}\label{e8.7}
		\lim_{n\to\infty}S_{n,b_n}(e^{ist};x)=W_{a/3}(e^{ist};x).
	\end{equation}
According to the definition we have
\begin{align}
	S_{n,b_n}(e^{ist};x)&=\dfrac{1}{(2nb_n)^n}\int_{x-nb_n}^{x+nb_n}\cdots\int_{x-nb_n}^{x+nb_n}e^{is(u_1+\dots+u_n)/n}du_1\dots du_n\\
	&=\dfrac{1}{(2nb_n)^n}\left(\int_{x-nb_n}^{x+nb_n}e^{isu/n} du\right)^n=e^{isx}\left(\dfrac{\sin sb_n}{sb_n}\right)^n.
\end{align}

Since $nb_n^2\to a$ we have $b_n\to 0$ and therefore
\begin{align*}
	\displaystyle\lim_{n\to\infty} S_{n,b_n}(e^{ist};x)&=e^{isx}\exp\left({\lim_{n\to\infty}\left(n\dfrac{\sin sb_n-sb_n}{sb_n}\right)}\right)\\
	&=e^{isx}\exp\left(s^2(nb_n^2)\dfrac{\sin sb_n-sb_n}{(sb_n)^3}\right).
\end{align*}
It follows that 
\begin{equation}
	\label{e8.8} \displaystyle\lim_{n\to\infty}S_{n,b_n}(e^{ist};x)=e^{isx-\frac{1}{6}as^2}.
\end{equation}
Now (\ref{e8.7}) is a consequence of (\ref{e8.4}) and (\ref{e8.8}). Thus (\ref{e8.6}) is proved.
	
\end{proof}

\section{The limit operator}\label{s12}

\begin{prop}
	Let $(L_n)_{n\geq 1}$ and $(T_m)_{m\geq 1}$ be two sequences of positive linear operators such that
	$$  \displaystyle \lim_{n\to\infty} L_{mn}(f(nt);\frac{x}{n})=T_m(f(t);x)$$
	for all $m\geq 1$, $x\geq 0$, $f\in C_b[0,\infty)$. Then
	\begin{equation}
		\label{e9.1} T_{m\nu}\left(f(\nu t);\frac{x}{\nu}\right)=T_m\left(f(t);x\right),
	\end{equation}
for all $m\geq 1$, $\nu\geq 1$, $x\geq 0$, $f\in C_b[0,\infty)$.
\end{prop}
\begin{proof} Let $\nu\geq 1$, $m\geq 0$. Then
	\begin{align*}
		&\displaystyle\lim_{n\to \infty}L_{m(n\nu)}\left(f((n\nu)t);\dfrac{x}{n\nu}\right)=T_m(f(t);x),\\
		&\displaystyle\lim_{n\to\infty}L_{(m\nu)n}\left(f(n(\nu t));\frac{x/\nu}{n}\right)=T_{m\nu}\left(f(\nu t);\frac{x}{\nu}\right)
	\end{align*}
and so we get (\ref{e9.1}).
\end{proof}

\begin{theorem}
	Let $0\leq a_0<a_1<\cdots$, $\varphi_{m,k}\in C_b[0,\infty)$, $m\in{\mathbb N}$, $k\in{\mathbb N}_0$,
	$$ T_m\left(f(t);x\right):=\displaystyle\sum_{k=0}^{\infty}\varphi_{m,k}(x)f\left(\frac{a_k}{m}\right),\,\,f\in C_b[0,\infty),\,\, x\geq 0.  $$
	Then (\ref{e9.1}) is equivalent to
	\begin{equation}
		\label{e9.2p} \varphi_{m\nu,k}(x)=\varphi_{m,k}(\nu x),\,\, m,\nu\in{\mathbb N}, \,\, k\in{\mathbb N}_0,\,\, x\geq 0,
	\end{equation}
and to
\begin{equation}
	\label{e9.3} \varphi_{m,k}(x)=\varphi_{1,k}(mx),\,\,m\in{\mathbb N},\,\, k\in{\mathbb N}_0,\,\, x\geq 0.
\end{equation}
\end{theorem}
\begin{proof}
	$T_{m\nu}\left(f(\nu t);\dfrac{x}{\nu}\right)=\sum_{k=0}^{\infty}\varphi_{m\nu,k}\left(\frac{x}{\nu}\right)f\left(\dfrac{a_k}{m}\right),$
	and it readily follows that (\ref{e9.1}) is equivalent to (\ref{e9.2p}). Suppose that (\ref{e9.3}) holds true. Then
	$$
\varphi_{m\nu,k}(x)=\varphi_{1,k}(m\nu x) \textrm{ and } \varphi_{m,k}(\nu x)=\varphi_{1,k}(m\nu x).$$
This implies (\ref{e9.2p}).

Finally, suppose that (\ref{e9.2p}) is valid. Using it with $m=1$ it follows that $\varphi_{\nu,k}(x)=\varphi_{1,k}(\nu x)$ for all $\nu \geq 1$, $k\in {\mathbb N}_0$, $x\geq 0$, and this (\ref{e9.3}).
\end{proof}

The Jakimovski-Leviatan operators are generalization of Sz\'asz-Mirakjan operators (see \cite{Leviatan1}, \cite{Leviatan2}). To define them we need a sequence of Appell polynomials $(A_n(x))_{n\geq 0}$. If they have the explicit representation

\begin{equation}\label{e2.1}
	A_n(x)=\alpha_n+{n\choose 1}\alpha_{n-1} x+{n\choose 2}\alpha_{n-2} x^2+\cdots +\alpha_0 x^n,\,\, n=0,1,\dots,
\end{equation}
one associates  the power series
\begin{equation}\label{e2.2}
	a(t)=\alpha_0+\dfrac{t}{1!}\alpha_1+\dfrac{t^2}{2!}\alpha_{2}+\cdots+\dfrac{t^n}{n!}\alpha_n+\cdots,\,\,\alpha_0\ne 0,
\end{equation}
called the generating function.

In fact, the sequence and its generating function are related by
\begin{equation}
	\label{e2.3}
	a(t)e^{tx}=A_0(x)+\dfrac{t}{1!}A_1(x)+\dfrac{t^2}{2!}A_2(x)+\cdots+\dfrac{t^n}{n!}A_n(x)+\cdots
\end{equation}
For details see, e.g.,
\cite{2}.

The Jakimovski-Leviatan type operators associated with $(A_n(x))_{n\geq 0}$ are defined by
\begin{equation}\label{e3.2} \Psi_n f(x):=\dfrac{e^{-nx}}{a(1)}\displaystyle\sum_{k=0}^{\infty}\dfrac{1}{k!}A_k(nx)f\left(\dfrac{k}{n}\right),\,\, n\geq 1, \end{equation}
where $f\in C[0,\infty)$ is a function for which the series is convergent for all $x\in [0,\infty)$.

\begin{remark}
	The Sz\'asz-Mirakyan and more generally the Jakimovski-Leviatan operators satisfy the equivalent conditions (\ref{e9.1}), (\ref{e9.2p}), (\ref{e9.3}). The Sz\'asz-Mirakyan operator appeared as limit operator in previous examples. The next example presents an instance where a Jakimovski-Leviatan operator appears as limit operator.
\end{remark}
\begin{exam}
	  Let $p\in [0,\infty)$ and $a(t):=e^{pt}$, $t\in{\mathbb R}$. Using (\ref{e2.3}) we find the associated Appell polynomials $A_{k,p}(x)=(x+p)^k,\, k\geq 0$. The Jakimovski-Leviatan operators are
	$$ \Psi_{n,p}(f(t);x)=e^{-nx-p}\displaystyle\sum_{k=0}^{\infty}\dfrac{(nx+p)^k}{k!}f\left(\dfrac{k}{n}\right),\,\, f\in C_b[0,\infty).  $$
	Let us remark that $\Psi_{n,px}\left(f(t);x\right)$ are the Sz\'asz-Mirakyan-Schurer operators presented, e.g. in \cite[p.338]{FA-MC}.
	
	Let $B_n^{[1]}$ be the classical Baskakov operators, and 
	$$L_n\left(f(t);x\right):=B_{mn}^{[1]}\left(f(nt);\dfrac{x+p/m}{n}\right).  $$
	Then $\displaystyle\lim_{n\to\infty} L_n(e^{ist};x)=e^{(mx+p)\left(e^{is/m}-1\right)}=\Psi_{m,p}(e^{ist};x)$ and consequently
	\begin{equation}\label{e.Psi}\displaystyle\lim_{n\to\infty}B_{mn}^{[1]}\left(f(nt);\dfrac{x+p/m}{n}\right)=\Psi_{m,p}(f(t);x). \end{equation}
\end{exam}

\begin{exam}
	Consider the operator
	$$ L_m^*\left(f(t);x\right):=\dfrac{1}{\cosh(mx)}\displaystyle\sum_{k=0}^{\infty} \dfrac{(mx)^{2k}}{(2k)!}f\left(\dfrac{2k}{m}\right),\,\, f\in C_b[0,\infty),\,\, x\geq 0, $$
	introduced by  Le\'sniewicz and Rempulska \cite{LR}; see also \cite{Ciupa}. It satisfies the equivalent conditions (\ref{e9.1}), (\ref{e9.2p}), (\ref{e9.3}).
	
	Now let $\tilde{B}_n\left(f(t);x\right):=\displaystyle\sum_{k=0}^n{n\choose k}(-x)^k(1-x)^{n-k}f\left(\dfrac{k}{n}\right),\,\, f\in C[0,1],\,\, x\in[0,1]$, and
	$$ L_n\left(f(t);x\right):=\dfrac{B_n^{[-1]}(f(t);x)+\tilde{B}_n(f(t);x)}{1+(1-2x)^n}, $$ 
	where
	$$ L_{2p+1}\left(f(t);1\right):=\displaystyle\lim_{x\to 1} L_{2p+1}\left(f(t);x\right)=f\left(\dfrac{2p}{2p+1}\right). $$
	Then $L_n$ reproduces the constant functions and 
	$$ \displaystyle\lim_{n\to\infty} L_{mn}\left(e^{isnt};\dfrac{x}{n}\right) =\dfrac{\cosh (mxe^{is/m})}{\cosh(mx)}=L_m^*\left(e^{ist};x\right). $$
	Consequently,
	\begin{equation}\label{e.Lm} \displaystyle\lim_{n\to\infty} L_{mn}\left(f(nt);\dfrac{x}{n}\right)=L_m^*\left(f(t);x\right), f\in C_b[0,\infty).  \end{equation}
\end{exam}

\begin{prop}
	Let $(L_m)_{m\geq 1}$ and $(Z_m)_{m\geq 1}$ be two sequences of positive linear operators such that
	$$\displaystyle\lim_{n\to\infty} L_m\left(f\left(\frac{t}{n}\right);nx\right)=Z_m(f(t),x)  $$
	for all $m\geq 1$, $x\geq 0$, $f\in C_b[0,\infty)$. Then
	\begin{equation}
		\label{e9.4} Z_m\left(f\left(\frac{t}{\nu}\right);\nu x\right)=Z_m\left(f(t),x\right),\,\, \nu\in{\mathbb N},\,\, x\geq 0.
	\end{equation}
\end{prop}
\begin{proof}We have
	$$ \displaystyle\lim_{n\to\infty}L_m\left(f\left(\dfrac{t/\nu}{n}\right);n(\nu x)\right)=Z_m\left(f\left(\dfrac{t}{\nu}\right);\nu x\right), $$
	and 
	$$ \displaystyle\lim_{n\to\infty}L_m\left(f\left(\dfrac{t}{n\nu}\right);(n\nu) x\right)=Z_m\left(f\left(t\right); x\right). $$
	Then leads to (\ref{e9.4}).
	\end{proof}
\begin{theorem}
	\label{T9.2} Let $Z_m$ be an integral operator of the form 
	$$ Z_m\left(f(t);x\right):=\displaystyle\int_0^{\infty}f(t)K_m(t,x)dt,\,\, f\in C_b[0,\infty),\,\, x\geq 0, $$
	with a suitable continuous kernel $K_m(t,x)$. Then (\ref{e9.4}) is equivalent to 
	\begin{equation}
		\label{e9.5} K_m(\nu t,\nu x)=\dfrac{1}{\nu}K_m(t,x),\,\, \nu>0,\,\, t,x\geq 0.
	\end{equation}
\end{theorem}
\begin{proof}
	We have 
	\begin{align*}
		Z_m\left(f\left(\dfrac{t}{\nu}\right);\nu x\right)&=\displaystyle\int_0^\infty f\left(\frac{t}{\nu}\right)K_m(t,\nu x)dt\\
		&=\nu\int_0^{\infty}f(t)K_m(\nu t,\nu x) dt.
	\end{align*}
So (\ref{e9.4}) becomes 
$$\displaystyle\int_0^{\infty}f(t)K_m(\nu t,\nu x)dt=\dfrac{1}{\nu}\int_0^{\infty}f(t)K_m(t,x) dt, $$
for all $f\in C_b[0,\infty)$, $\nu>0$, $x>0$. This is equivalent to (\ref{e9.5}) and the proof is finished.
\end{proof}
\begin{exam}\label{e9.2} For the Gamma operator the kernel is
	$$ K_m(t,x):=\dfrac{x^{-m}}{(m-1)!}m^mt^{m-1}e^{-mt/x}, $$
	and it satisfies (\ref{e9.5}).
	\end{exam}

\begin{exam}
	Let $R_m$ be an operator of the form
	$$ R_m(f(t);x):=\displaystyle\sum_{k=0}^{\infty}\varphi_{m,k}(x)\frac{\int_0^{\infty}\varphi_{m,k}(t)f(t)dt}{\int_0^{\infty}\varphi_{m,k}(t)dt}.  $$
\end{exam}

If the functions $\varphi_{m,k}$ satisfy (\ref{e9.2p}), then the operators $R_m$ satisfy (\ref{e9.1}).

\begin{problem}\label{P9.1}
Is the converse true?
\end{problem}
\begin{comment} Suppose that $(R_m)_{m\geq 1}$ satisfies (\ref{e9.1}) for all $m\geq 1$, $\nu \geq 1$, $x\geq 0$, $ f\in C_b[0,\infty)$. With $m=1$ we get $R_{\nu}\left(f(\nu t);\dfrac{x}{\nu}\right)=R_1\left(f(t);x\right)$. With an obvious change of notation we have $R_m(f(s);y)=R_1\left(f\left(\dfrac{s}{\nu}\right);\nu y\right)$. Even better, 
\begin{equation}
	\label{e9.6} R_m\left(f(t);x\right)=R_1\left(f\left(\dfrac{t}{m}\right);mx\right), \textrm{ for all } m\geq 1, \,\, x\geq 0.
\end{equation}
Therefore each $R_m$, $m\geq 2$, is completely determined by $R_1$.

 A natural hypothesis is  $\displaystyle \int_0^{\infty}\varphi_{m,k}(x)dx=1/m$, for all $m\geq 1$, $k\geq 0$. Then (\ref{e9.6}) becomes
 
 \begin{equation}
 	\label{e9.7} \displaystyle\sum_{k=0}^{\infty}\varphi_{m,k}(x)\int_0^{\infty}\varphi_{m,k}(t)f(t)dt=\sum_{k=0}^{\infty}\varphi_{1,k}(mx)\int_0^{\infty}\varphi_{1,k}(mt)f(t) dt,
 \end{equation}
for all $m\geq 1$, $x\geq 0$, $f\in C_b[0,\infty)$.

In this context Problem \ref{P9.1} is reduced to: does (\ref{e9.7}) imply  (\ref{e9.3})?

Generally speaking, if a Durrmeyer operator is represented as
$$ \displaystyle\sum_{k=0}^{\infty} u_k(x)\int_0^{\infty}u_k(t)f(t)dt=\sum_{k=0}^{\infty}v_k(x)\int_0^1v_k(t)f(t)dt,  $$
for all $x\geq 0$ and  $f\in C_b[0,\infty)$, does it follow that $u_k=v_k$, $k\geq 0$?
\end{comment}

\begin{prop} The following two statements are equivalent

\begin{itemize}
	\item[(1)] $R_{m\nu}\left(f(\nu t);\dfrac{x}{\nu}\right)=R_m\left(f(t);x\right), \,\, \forall\,\, m,\nu\geq 1,\,\, x\geq 0,\,\, f\in C_b[0,\infty).$\\
	\item[(2)]$R_m\left( f(t); x\right)=R_1\left(f\left(\dfrac{t}{m}\right);mx\right),\,\,\forall \,\, m\geq 1,\,\, x\geq 0,\,\, f\in C_b[0,\infty)$.
\end{itemize}
\end{prop}

\begin{proof} (1) $\Rightarrow $ (2). Take $m=1$. Then $R_{\nu}\left(f(\nu t); \dfrac{x}{\nu}\right)=R_1\left(f(t);x\right)$. Set $s=\nu t$, $y=\dfrac{x}{\nu}$.

Then $R_{\nu}\left(f(s);y\right)=R_1\left(f\left(\dfrac{s}{\nu}\right);\nu y\right)$, which is (2).

(2) $\Rightarrow $ (1). Using (2) we have \begin{align*}R_{m\nu}\left(f(\nu t);\dfrac{x}{\nu}\right)&=R_1\left(f\left(\dfrac{\nu t}{m\nu}\right);m\nu\dfrac{x}{\nu}\right)\\
	&=R_1\left(f\left(\dfrac{t}{m}\right);mx\right)=R_m\left(f(t);x\right),\end{align*} and (1) is proved.

\end{proof}

\begin{exam}
	We try to determine a convolution operator $U(f(t);x)=\displaystyle\int_{\mathbb R}f(t)\varphi(x-t) dt$ satisfying (\ref{e9.4}) for all $\nu\in(0,\infty)$.
	
	$$ U\left(f\left(\dfrac{t}{\nu}\right);\nu x\right)=\displaystyle\int_{\mathbb R}f\left(\dfrac{t}{\nu}\right)\varphi(\nu x-t) dt=\int_{\mathbb R}f(s)\varphi(\nu x-\nu s)\nu ds $$
	and so $ \varphi\left(\nu(x-s)\right)=\dfrac{1}{\nu}\varphi(x-s)$, i.e., $\varphi(\nu z)=\dfrac{1}{\nu}\varphi(z)$.
	
	Taking $ \varphi(z):=\dfrac{1}{\pi z}$ we see that $Uf$ is the convolution of $f(t)$ and $\dfrac{1}{\pi t}$, i.e., the Hilbert transform of $f$. But, of course, this is not a positive linear operator.
\end{exam}

We end this section with a list with the limit operators prezented  in the previous sections, together with an equation where they appear: $B_m^{[0]}$ (\ref{e***1}), $V_{m}^{[0](k)}$ (\ref{e4.4}), $D_{m,\rho}^{[0](k)}$ (\ref{e5.3}),\linebreak  $V_m^{[0](k)}\left((f\circ \sigma)(t);\dfrac{\psi(x)}{n}\right)$ (\ref{e.Maja}), $G_{\frac{m}{n}+k}$ (\ref{e.Gamma}), $W_a$ (\ref{e8.1}), $W_{a/3}$ (\ref{e8.6}), $\Psi_{m,p}$ (\ref{e.Psi}), $L_m^*$ (\ref{e.Lm}),\linebreak  $G_{m+1}$ (\ref{e.Gamma1}).

$  $

\noindent{\bf Funding.} 
 This work has been supported by the University of Wuppertal.

$  $

\noindent{\bf Conflicts of interest.} The authors declare no competing financial interests.

$  $

\noindent{\bf Availability of data and material.} No data were used to support this study.

$  $


\end{document}